\documentclass{amsart}
\usepackage[a4paper,margin=2.5cm]{geometry}
 
\usepackage[T1]{fontenc}
\usepackage{amssymb, amsmath, amsthm}
\usepackage{enumitem}
\usepackage{float}
﻿
﻿
\usepackage{url}

\usepackage[backend=bibtex,style=alphabetic,sorting=nyt,isbn=false,url=false,doi=true,maxalphanames=10,minalphanames=4,mincitenames=4,maxcitenames=10,minnames=4,maxnames=10,giveninits=true,maxbibnames=99]{biblatex}
﻿
\usepackage{xcolor}
\usepackage{hyperref}

﻿
﻿
﻿
\usepackage{physics}
\usepackage{amsmath}
\usepackage{tikz}
\usepackage{mathdots}
\usepackage{yhmath}
\usepackage{cancel}
﻿
\usepackage{siunitx}
\usepackage{array}
\usepackage{multirow}
\usepackage{amssymb}
\usepackage{gensymb}
\usepackage{tabularx}
\usepackage{extarrows}
\usepackage{booktabs}
\usetikzlibrary{fadings}
\usetikzlibrary{patterns}
\usetikzlibrary{shadows.blur}
\usetikzlibrary{shapes}
\usetikzlibrary{decorations.pathreplacing}
\usepackage{relsize}

\usepackage[all]{xy}
﻿
\usepackage{mathtools}

﻿
\DeclareFontFamily{U}{stix2bb}{}
\DeclareFontShape{U}{stix2bb}{m}{n}{<->stix2-mathbb}{}
﻿
﻿
﻿
\numberwithin{equation}{section} 
﻿
﻿
\theoremstyle{plain}
\newtheorem{theorem}{Theorem}[section]
\newtheorem{prop}[theorem]{Proposition}
\newtheorem{lemma}[theorem]{Lemma}

﻿
﻿
\theoremstyle{remark}
\newtheorem{remark}[theorem]{Remark}

﻿
﻿
\def\cxi{\mathbf{i}}
﻿
﻿
\title{The quantum integrable hierarchy for the Gromov-Witten theory of~ elliptic~curves}
﻿
\author{Paolo Rossi}
\address{P. R.: Dipartimento di Matematica ``Tullio Levi-Civita'', Università degli studi di Padova,
Via Trieste 63, 35121 Padova, Italia}
﻿\email{paolo.rossi@math.unipd.it}

\author{Sergey Shadrin}
﻿\address{S. S.: Korteweg-de Vries Instituut voor Wiskunde, Universiteit van Amsterdam, Postbus 94248, 1090GE Amsterdam, The Netherlands}
\email{s.shadrin@uva.nl}

\author{Ishan Jaztar Singh}
\address{I. J. S.: Dipartimento di Matematica ``Tullio Levi-Civita'', Università degli studi di Padova, Via Trieste 63, 35121 Padova, Italia}﻿
\email{i.s.ishansingh16@gmail.com}
﻿
\begin{document}
﻿
\begin{abstract}
We construct the quantum double ramification hierarchy associated with the Gromov-Witten theory of elliptic curves. We use results of Oberdieck and Pixton on the intersection numbers of the double ramification cycle, the Gromov-Witten classes of the elliptic curve and the Hodge class $\lambda_{g-1}$ together with vanishing results for $\lambda_{g-2}$ to produce a closed, modular expression for the resulting integrable hierarchy. It is the first explicit nontrivial example of a quantum integrable hierarchy from a cohomological field theory containing fermionic fields, which correspond to the odd classes in the cohomology of the elliptic curve.
\end{abstract}
﻿
\maketitle
﻿
\tableofcontents
﻿
\section*{Introduction}

Gromov-Witten theories of target varieties with non trivial odd cohomology require a superalgebraic version of the notion of cohomological field theory \cite{kontsevich_manin_1994} and lead to correspondingly supergeometric constructions of the associated integrable hierarchies, giving rise to integrable systems with both bosonic and fermionic dynamical fields. In the literature the super versions of these construction are often hinted at, but rarely developed in detail, in no small measure due the lack of computed examples (see, however,~\cite{DSVV}).\\

The elliptic curve provides the simplest example of projective variety with nontrivial odd cohomology. The Gromov-Witten theory of curves was solved in \cite{okounkov_pandharipande_06_a,okounkov_pandharipande_06_b, okounkov_pandharipande_06_c}, but for the elliptic curve the corresponding integrable hierarchy is somewhat trivial as, up to coordinate change, it reduces to a purely dispersionless system of evolutionary PDEs  \cite{buryak_22}. Incidentally, all curves of positive genus, as any target variety with non positive first Chern class, have a quite simple associated classical hierarchy \cite{buryak_dubrovin_guéré_rossi_18}.\\

The double ramification hierarchy of \cite{buryak_15}, one of the two extant constructions (the other one being the Dubrovin-Zhang hierarchy of \cite{dubrovin_zhang_01,buryak_posthuma_shadrin_12}) of an integrable hierarchy of evolutionary PDEs from a given cohomological field theory, employs the intersection theory of the latter with the double ramification cycle, the top Hodge class and the psi classes. In \cite{buryak_rossi_15} it was shown to admit a quantization as a Hamiltonian system, a deformation involving intersection numbers with the rest of the Hodge Chern polynomial. This quantization is a priori nontrivial for the elliptic curve.\\

In this paper we compute the quantum DR hierarchy of (the Gromov-Witten theory of the) elliptic curve. Dimension counting leaves room for exactly two terms in the $\hbar$-deformation of the primary Hamiltonians of the system: for the term linear in $\hbar$ the relavant intersection numbers were computed in \cite{pixton_oberdieck_23}, while we show that the quadratic term in $\hbar$ has to vanish a priori (by a direct geometric argument: the pushforward along a branched covering map $C_1 \to C_2$ between Riemann surfaces sends principal divisors to principal divisors, so, if $C_1$ has positive genus, you cannot completely constrain the image in $C_2$ of the support of such divisor because it will generically violate the Jacobian conditions).\\

Thanks to the computations of \cite{pixton_oberdieck_23} the resulting quantum DR hierarchy exhibits a quasi-modular nature (in the degree parameter). In particular the closed expression for the primary Hamiltonians encodes the dynamics of two bosonic and two fermionic fields interacting via the integral transform of a quadratic interaction term involving shift operators, with kernel depending modularily on one of the bosonic fields themselves. This is different from what happens in more traditional ellptic quantum integrable systems, where the modular parameter is fixed, not a dynamical variable.\\

Beside the ordinary classical limit of \cite{buryak_22}, one can perform a simple cusp limit in the modular parameter, obtaining a trigonometric kernel instead, or a more sophisticated double scaling limit to produces a classical dispersive nonlinear system. Finally, the dispersionless limit of the full quantum system has a primary Hamiltonian whose modular dependence is exclusively via the $\tau$-derivative of the Eisenstein series of quasimodular degree $2$.\\

Several questions remain about this integrable Hamiltonians, the main one being about their nontriviality with respect to changes of coordinates: is there a quantum analogue of the Miura triviality observed for the classical limit? If not, can we identify a known integrable model corresponding to either our full DR hierarchy or one of its simpler limits?\\
 
\subsection*{Conventions and Notations}
\begin{itemize}
    \item The Einstein summation convention is applied for repeated upper and lower Greek indices.
    \item The symbol $*$ is used to represent any value within the appropriate range of a subscript or superscript.
    \item For a given topological space $X$, denote by $H_*(X)$ and $H^*(X)$ the homology and cohomology groups of $X$ with coefficients in $\mathbb{C}$.
    \item The moduli space of stable curves is denoted by $\overline{M}_{g,n}$.
    \item The ring of integers modulo $k$ is denoted by $\mathbb{Z}_k$.
    \item The symbol \( i \) is used in two different contexts: in generating sums of quantum hierarchies, it denotes the complex number \(\sqrt{-1}\); in other contexts, such as subscripts and combinatorial sums, it typically represents a positive integer.
\end{itemize}
﻿
\subsection{Acknowledgements}
P. R. is supported by the University of Padova and
is affiliated to the INFN under the national project MMNLP and to the INdAM group GNSAGA. S. S. is supported by the Netherlands Organization for Scientific Research. I. J. S. is supported by the University of Padova and the Marie Curie Fellowship (project ID 741896) and is affiliated with the INFN under the national project MMNLP. Additionally, we extend our gratitude to the Korteweg-de Vries Institute for Mathematics at the University of Amsterdam for hosting I. J. S. during the completion of of this work.
 
\newpage 
﻿
﻿
\section{(Super) quantum double ramification hierarchy}
﻿
The purpose of this section is to recall the definition of the quantum double ramification hierarchy~\cite{buryak_rossi_15,buryak_dubrovin_guéré_rossi_20}, pointing out the adjustments needed for the $\mathbb{Z}_2$-graded case.
﻿
﻿
\subsection{Cohomological field theories}
\label{section: CohFT}
Let $V$ be a finite dimensional $\mathbb{C}$-vector space spanned by generators $\{e_1,\ldots,e_N\}$ where $e_1$ is a distinguished unit element. Moreover, let $V$ be $\mathbb{Z}_2$-graded, that is, $V=V_0\oplus V_1$, where elements of $V_0$ are called \textit{even} vectors, while elements of $V_1$ are called \textit{odd} vectors. We assume that $e_1\in V_0$. It is convenient to use the map $\deg_{\mathbb{Z}_2}\colon V \to \mathbb{Z}_2$ that maps the even vectors to $0$ and the odd vector to $1$. Equip $V$ with a non-degenerate bilinear form $\eta:V\times V \rightarrow \mathbb{C}$, which is even and graded symmetric. The latter requirements mean that
\begin{itemize}
    \item $\eta$ restricted to $V_0\otimes V_0$ is symmetric;
    \item $\eta$ restricted to $V_1\otimes V_1$ is skew-symmetric;
    \item $\eta$ restricted to $V_0\otimes V_1$ and $V_1\otimes V_0$ vanishes. 
\end{itemize}
A \textit{cohomological field theory (CohFT)} is a family of maps $\{c_{g,n}\}_{2g-2+n>0}$, where
\begin{equation} \label{eq:definition-CohFT-cgn}
c_{g,n}\colon V^{\otimes n}\rightarrow H^{*}(\overline{M}_{g,n},\mathbb{C}), \qquad 
v_1\otimes \dots \otimes v_n \mapsto c_{g,n}(v_1\otimes \dots \otimes v_n).
\end{equation}
We can also identify $c_{g,n}$ as elements of $H^{*}(\overline{M}_{g,n},\mathbb{C})\otimes  (V^*)^{\otimes n}$, and we refer to them as `classes'. The classes $c_{g,n}$ are required to satisfy the following properties: 
﻿
\begin{enumerate}
\item[i.] The classes $c_{g,n}$ are even, i.e., the maps~\eqref{eq:definition-CohFT-cgn} preserve the $\mathbb{Z}_2$-grading.  
\item[ii.] The classes $c_{g,n}$ are graded equivariant under the action of the symmetry group $S_n$ by permutations on the factors of $V^{\otimes n}$ and by relabeling of the marked points on $\overline{M}_{g,n}$. 
This means that for a given permutation $s\in S_n$ we denote $\sigma_s\colon \overline{M}_{g,n}\to \overline{M}_{g,n}$ to be the isomorphism induced by the relabeling of the marked points according to $s$, and then
\begin{equation*}
\pm c_{g,n}(v_{s(1)}\otimes\ldots\otimes v_{s(n)})) = (\sigma_s^{-1})^*c_{g,n}(v_1\otimes\ldots\otimes v_n),
\end{equation*} 
where the sign $\pm$ is the Koszul sign corresponding to the reordering of the homogeneous vectors $v_i$ by the permutation $s$. 
﻿
\item[iii.] Consider the gluing map of the first type,
\begin{equation*}
\mathsf{gl_1}\colon \overline{M}_{g_1,n_1+1}\times\overline{M}_{g_2,n_2+1}\rightarrow \overline{M}_{g,n},
\end{equation*}
where $g=g_1+g_2$ and $n=n_1+n_2$, and we assume that the points with the labels $i\in S_1$ (resp., $i\in S_2$) lie on the first (resp., second) component of the degenerated curve, so $S_1\sqcup S_2 = \{1,\dots,n\}$ and $|S_1|=n_1$ and $|S_2|=n_2$. Then its pullback on $c_{g,n}$ is given by
\begin{align*}
& \mathsf{gl_1}^*(c_{g,n}(v_1,\dots,v_n))= 
\pm c_{g_1,n_1+1}\big( e_\alpha \otimes \bigotimes_{i\in S_1} v_i \big)\eta^{\alpha\beta}c_{g_2,n_2+1}\big(e_\beta\otimes \bigotimes_{i\in S_2} v_i\big),
\end{align*}
where the sign $\pm$ is the Koszul sign. 
\item[iv.] For the gluing map of the second type,
\begin{equation*}
\mathsf{gl_2}\colon \overline{M}_{g-1,n+2}\rightarrow \overline{M}_{g,n},
\end{equation*}
its corresponding pullback on $c_{g,n}$ is expressed as
\begin{equation*}
\mathsf{gl_2}^*c_{g,n}(v_1\otimes \dots \otimes v_n))=c_{g-1,n+2}(v_1\otimes \dots\otimes v_{n}\otimes e_\alpha \otimes e_\beta)\eta^{\alpha\beta}.
\end{equation*}
Note that there is no Koszul sign in this case since $\eta$ is even. 
\item[v.] For the forgetful map that forgets the last marked point,
\begin{equation*}
\mathsf{fg}\colon \overline{M}_{g,n+1}\rightarrow \overline{M}_{g,n},
\end{equation*}
its corresponding pullback on $c_{g,n}$ is expressed as
\begin{equation*}
\mathrm{fg}^*(c_{g,n}(v_1\otimes \dots\otimes v_n))=c_{g,n+1}(v_1\otimes \dots\otimes v_n\otimes e_1)
\end{equation*}
Note that there is no Koszul sign in this case since $e_1$ is even. We also demand that 
\begin{equation*}
\eta(v_1,v_2)=c_{0,3}(v_1\otimes v_2\otimes e_1).
\end{equation*}
\end{enumerate}
﻿
\subsection{Double ramification cycle}
﻿
\label{sec: DR definition}
﻿
Let $A = (a_1, \ldots, a_n) \in \mathbb{Z}^n$ be a tuple such that $\sum_i a_i = 0$. Let $A_+$ denote a subtuple of $A$ consisting of all the positive integers in $A$, and let $A_-$ denote a subtuple of $A$ consisting of all the negative integers in $A$, and let $n_0$ be the number of $a_i$'s that are equal to $0$. Let $\overline{M}^\sim_{g,n_0}(\mathbb{P}^1, A_-, A_+)$ be the moduli space of stable relative maps of connected genus $g$ curves to the rubber, with ramification profiles $A_-$ and $A_+$ over the points $0, \infty \in \mathbb{P}^1$, respectively. Let
\begin{equation*}
	\mathsf{src} \colon \overline{M}^\sim_{g,n_0}(\mathbb{P}^1, A_-, A_+) \rightarrow \overline{M}_{g,n}
\end{equation*}
be the source map that forgets the stable relative map and retains the stabilization of the source curve. Moreover, the space $\overline{M}^\sim_{g,n_0}(\mathbb{P}^1, A_-, A_+)$ is endowed with the virtual fundamental class, whose Poincar\'e dual is denoted by
\begin{equation*}
	\left[\overline{M}^\sim_{g,n_0}(\mathbb{P}^1, A_-, A_+)\right]^{\mathrm{vir}} \in H^{2g}(\overline{M}^\sim_{g,n_0}(\mathbb{P}^1, A_-, A_+)).
\end{equation*}
Define the double ramification cycle $\mathrm{DR}_g(A)$ for a given tuple $A$ as:
\begin{equation*}
	\mathrm{DR}_g(A) \coloneqq \mathsf{src}_*\left[\overline{M}^\sim_{g,n_0}(\mathbb{P}^1, A_-, A_+)\right]^{\mathrm{vir}} \in H^{2g}(\overline{M}_{g,n}).
\end{equation*}
﻿
﻿
The double ramification cycle $\mathrm{DR}_g(A)$ can be described as a restriction to the hyperplane $\sum_i a_i=0$ an even polynomial of degree $2g$ in the parameters $a_i$ with the coefficients in the tautological classes in $H^{2g}(\overline{M}_{g,n})$. This description is not unique (as we can add an arbitrary polynomial that vanishes on the hyperplane $\sum_i a_i=0$); so we fix it in the way that doesn't depend a particular variable $a_j$, $j=1,\dots,n$
\begin{equation}
\label{eq:nonSymmetricPolyDR}
	\operatorname{DR}_{g}(A)=\sum_{\footnotesize\substack{K=(k_1,\dots,k_n) \\ k_j = 0 \\ k_i \in \mathbb{Z}_{\geq 0}, i\not=j \\ k_1+\cdots+k_n\leq 2g}} \mathfrak{D}^{(j)}_{g}(K)\prod_{i=1}^n a_i^{k_i}.
\end{equation}
The average of these expressions over $j=1,\dots,n$ gives us another expression for the double ramification cycle, denoted by
\begin{equation} \label{eq:SYmmetricPolynomialDR}
	\textrm{DR}_{g}(A)=\sum_{\footnotesize\substack{K=(k_1,\dots,k_n) \\ k_i \in \mathbb{Z}_{\geq 0} \\ k_1+\cdots+k_n\leq 2g}} \mathfrak{D}^{\mathrm{sym}}_{g}(K)\prod_{i=1}^n a_i^{k_i},
\end{equation}
which is a polynomial symmetric all variables $a_1,\dots,a_n$: 
﻿
﻿
\subsection{Quantum commutator and local functionals}
﻿
Let $V$ be the $N$-dimensional graded $\mathbb{C}$-vector space considered in Section \ref{section: CohFT}. We associate to it the ring of \emph{quantum differential polynomials} defined as the $\mathbb{Z}_2$-graded ring $\mathcal{A}=\mathbb{C}[\![ u^{\alpha} ]\!][u^{\alpha}_{j}][\![\varepsilon,\hbar]\!]$ equipped with an additional $\partial_x$-gradation $\deg_{\partial_x}\colon \mathcal{A} \to \mathbb{Z}$,  where
\begin{itemize}
    \item $u^{\alpha}_{j}$ are formal variables indexed by $\alpha = 1, \dots, N$ and $j \in \mathbb{Z}_{\geq 0}$ with the $\partial_x$-degree $\deg_{\partial_x}u^{\alpha}_{j}=j$;
    \item $\varepsilon$ is the dispersion parameter with $\partial_x$-degree $\deg_{\partial_x}(\varepsilon) = -1$;
    \item $\hbar$ is the quantization parameter with $\partial_x$-degree $\deg_{\partial_x}(\hbar) = -2$.
\end{itemize}
The $\mathbb{Z}_2$-grading $\mathcal{A} = \mathcal{A}_0 \oplus \mathcal{A}_1$ descents from the $\mathbb{Z}_2$-grading on $V$. That is, we assume that the basis $\{e_\alpha\}$ of $V$ is homogeneous with respect to the $\mathbb{Z}_2$-grading, and then defined $\deg_{\mathbb{Z}_2}\colon \mathcal{A}\to \mathbb{Z}_2$ as a ring morphism with the following initial values of $\deg_{\mathbb{Z}_2}$ on the generators: 
\begin{itemize}
	\item $\deg_{\mathbb{Z}_2}(u^\alpha_j) \coloneqq \deg_{\mathbb{Z}_2}(e_\alpha)$;
	\item $\deg_{\mathbb{Z}_2}(\epsilon) \coloneqq 0$;  
	\item $\deg_{\mathbb{Z}_2}(\hbar) \coloneqq 0$.
\end{itemize}
﻿
﻿
We have the following even differential operator on $\mathcal{A}$ of $\deg_{\partial_{x}}$-degree $+1$:
\begin{equation*}
    \partial_x = \sum_{i\geq 0} u^{\alpha}_{i+1}\frac{\partial}{\partial u^\alpha_i}.
\end{equation*}
Note that in particular $u^\alpha_j = \partial_x^j u^\alpha$.
﻿
﻿
﻿
﻿
We define the space of \emph{quantum local functionals} as the quotient space $\mathcal{A}/ (\mathrm{Im}(\partial_x) \oplus \mathbb{C}[\![\varepsilon,\hbar]\!])$, where $\mathbb{C}$ is the subspace of constants in $u^{\alpha}_j$ in $\mathcal{A}$. The projection operator from the ring of differential polynomials modulo the constants to the space of local functionals is given by:
\begin{align*}
 \mathcal{A}/\mathbb{C}[\![\varepsilon,\hbar]\!] & \rightarrow \mathcal{A}/ \left(\textrm{Im}(\partial_x)\oplus \mathbb{C}[\![\varepsilon,\hbar]\!]\right); \\
f & \mapsto \overline{f}\coloneqq \int f\, dx.
\end{align*}
﻿
We often use another set of formal variables $p^{\alpha}_{k}$ that are related to the $u$-variables via a Fourier series expansion,
\begin{equation}
\label{eqn: fourier series expansion}
    u^\alpha = \sum_{k \in \mathbb{Z}} p^\alpha_k e^{\cxi k x}.
\end{equation}
This allows us to consider elements of the ring $\mathcal{A}$ inside the auxiliary ring $\mathbb{C}[\![ p^{\alpha}_{k>0} ]\!][p^{\alpha}_{k\leq 0}][\![\varepsilon,\hbar]\!][e^{\cxi x}, e^{-\cxi x}]$.
Here $p^{\alpha}_{k}$ are formal variables indexed by $\alpha = 1, \ldots, N$ and $k \in \mathbb{Z}$, and $\deg_{\mathbb{Z}_2}  p^{\alpha}_{k} \coloneqq \deg_{\mathbb{Z}_2} u^\alpha$.  This auxiliary ring is used for construction of the $\star$ product and the quantum commutator as follows.
﻿
For any $f,g\in \mathcal{A}$ homogeneous with respect to the $\mathbb{Z}_2$-grading, the \textit{quantum commutator} is given in terms of the star product defined on $\mathcal{A}$:
\begin{equation}
\label{eqn: quantum commutator}
    [f,g] \coloneqq f\star g - \pm g\star f
\end{equation}
where the sign $\pm$ is the Koszul sign (which is equal in this case to $(-1)^{\deg_{\mathbb{Z}_2} (f) \cdot \deg_{\mathbb{Z}_2} (g)}$) and the $\star$ product is defined as
\begin{align*}
f\star g &=f \  \mathrm{exp}\left(\sum _{k\geq 0} \cxi\hbar k\eta ^{\alpha \beta }{\frac{\overleftarrow\partial }{\partial p_{k}^{\alpha }}}{\frac{\overrightarrow\partial }{\partial p_{-k}^{\beta }}}\right)  g \\
& =\sum _{ \footnotesize \substack{n\geq 0 \\ k_{1} ,k_{2} ,\ldots ,k_{n} \geq 0}}\pm \frac{( \cxi\hbar ){^{n}}}{n!} k_{1} \eta ^{\alpha _{1} \beta _{1}} \ldots k_{n} \eta ^{\alpha _{n} \beta _{n}}\frac{\partial ^{n} f}{\partial p_{k_{1}}^{\alpha _{1}} \ldots \partial p_{k_{n}}^{\alpha _{n}}}\frac{\partial ^{n} g}{\partial p_{-k_{1}}^{\beta _{1}} \ldots \partial p_{-k_{n}}^{\beta _{n}}},
\end{align*}
where $\pm$ is the Koszul sign. 
As a consequence, in the auxiliary ring we obtain the standard commutation rules for creation and annihilation operators,
\begin{equation*}
    \frac{1}{\cxi\hbar}[p^{\alpha}_{k},p^{\beta}_j] = k\eta^{\alpha \beta}\delta_{k+j,0}.
\end{equation*}
The quantum commutator can be explicitly expressed in the $u$-variables, as shown in \cite{buryak_rossi_15}. It defines a structure of the $\mathbb{Z}_2$-graded Lie algebra, that is, it is $\mathbb{Z}_2$-graded skew-symmetric and satisfies the $\mathbb{Z}_2$-graded Jacobi identity. 
﻿
﻿
\subsection{Hamiltonian densities}
The Hamiltonians $\overline{G}_{\alpha,d}$ of the quantum double ramification hierarchy are defined in terms of their densities,
\begin{equation*}
\overline{G}_{\alpha,d} = \int G_{\alpha,d}(x) \, dx,
\end{equation*}
and the latter can be expressed in terms of the \(p\)- and $u$-variables, respectively, as
\begin{align*}
G_{\alpha ,d} &= \sum _{\footnotesize\substack{g,n\in \mathbb{Z}_{\geq 0} \\ 2g-1+n>0}} \frac{(\cxi\hbar)^{g}}{n!} \sum_{\footnotesize\substack{A=(a_1,\dots,a_n)\\ a_i \in \mathbb{Z}
}} \int_{\overline{M}_{g,n+1}} \mathrm{DR}_{g}\left( -\textstyle\sum_{i=1}^n a_{i}, A\right)\, \Lambda( \tfrac{-\varepsilon^2}{\cxi\hbar}) \, \psi_{1}^{d}\, c_{g,n+1}\big(e_{\alpha} \otimes \bigotimes_{i=1}^n e_{\alpha_i} \big)  \prod_{i=1}^n p^{\alpha_i}_{a_i}e^{\cxi a_i x}  \\
&= \sum _{\footnotesize\substack{g,n\in \mathbb{Z}_{\geq 0} \\ 2g-1+n>0}} \frac{(\cxi\hbar)^{g}}{n!} \sum_{\footnotesize\substack{K=(0,k_1,\dots,k_n) \\ k_i \in \mathbb{Z}_{\geq 0} \\ k_1+\cdots+k_n\leq 2g}} 
\int_{\overline{M}_{g,n+1}}\mathfrak{D}^{(1)}_{g}(K) \, \Lambda( \tfrac{-\varepsilon^2}{\cxi\hbar}) \, \psi_{1}^{d}\, c_{g,n+1}\big(e_{\alpha} \otimes \bigotimes_{i=1}^n e_{\alpha_i} \big) \prod_{i=1}^n u^{\alpha_i}_{k_i} .
\end{align*}
Here $ \Lambda(\tfrac{-\varepsilon^2}{\cxi\hbar}) = \sum_{j=0}^g (\tfrac{-\varepsilon^2}{\cxi\hbar})^j \lambda_j$, where $\lambda_j=c_j(\mathbb{E})\in H^{2j}(\overline{M}_{g,n+1})$ denote the $j$-th Chern class of the Hodge bundle $\mathbb{E}$ over $\overline{M}_{g,n+1}$, $j=1,\dots,g$. 
﻿
\par

It is useful to separately define $G_{\alpha,-1} \coloneqq \eta_{\alpha\mu}u^{\mu}$, which gives the Casimir elements of the quantum commutator, and the local functional
\begin{align}\label{eq:DR potential}
\overline{G} = \int \left[ \sum _{\footnotesize\substack{g,n\in \mathbb{Z}_{\geq 0} \\ 2g-2+n>0}} \frac{(\cxi\hbar)^{g}}{n!} \sum_{\footnotesize\substack{K=(k_1,\dots,k_n) \\ k_i \in \mathbb{Z}_{\geq 0} \\ k_1+\cdots+k_n\leq 2g}} 
\int_{\overline{M}_{g,n}} \mathfrak{D}^{\mathrm{sym}}_{g}(K) \, \Lambda( \tfrac{-\varepsilon^2}{\cxi\hbar}) \, c_{g,n}\big(\bigotimes_{i=1}^n e_{\alpha_i} \big)  \prod_{i=1}^n u^{\alpha_i}_{k_i} \, \right]dx,
\end{align}
which we refer to as the DR hierarchy potential. This last local functional plays a role towards the DR hierarchy analogous to the role of the Dubrovin-Frobenius potential towards the principal hierarchy of the corresponding Frobenius manifold. In particular, pushing fowards with respect to the first marked point and using the dilaton equation, we obtain immediately
\begin{equation}\label{eq: dilaton}
\overline{G}_{1,1} = \Big( \varepsilon \frac{\partial}{\partial \varepsilon} + 2\hbar \frac{\partial}{\partial \hbar} + \sum_{s=0}^\infty u^{\alpha}_s \frac{\partial}{\partial u^{\alpha}_s}-2\Big) \overline{G}.
\end{equation}

All together, the quantum Hamiltonians form a quantum integrable system:
\begin{prop}[\cite{buryak_rossi_15}]\label{proposition:commutativity}
	For all $\alpha,\beta=1,\ldots,N$ and $p,q\in\mathbb{Z}_{\geq -1}$ we have:
	\begin{equation*}
		[\overline{G}_{\alpha ,q} ,\overline{G}_{\beta ,q}] =0.
	\end{equation*}
\end{prop}
﻿
Moreover, we can reconstruct the entire quantum double ramification hierarchy from the knowledge of $\overline{G}_{1,1}$ (and hence from $\overline{G}$) alone: 
\begin{prop}[\cite{buryak_rossi_15}] \label{prop:recursion relation for qHam}
For all $\alpha=1,\ldots,N$ and $p\in\mathbb{Z}_{\geq -1}$ we have:
\begin{equation*}
\partial _{x}\Big( \varepsilon \frac{\partial}{\partial \varepsilon} + 2\hbar \frac{\partial}{\partial \hbar} + \sum_{s=0}^\infty u^{\alpha}_s \frac{\partial}{\partial u^{\alpha}_s}-1\Big) G_{\alpha ,p+1} =\frac{1}{\hbar }[ G_{\alpha ,p} ,\overline{G}_{1,1}].
\end{equation*}
\end{prop}
﻿Notice how, in light of Proposition \ref{proposition:commutativity}, the right-hand side of the equation in Proposition \ref{prop:recursion relation for qHam} lies in the image of the operator $\partial_x$, and how the operator in parenthesis in the left-hand side of the above equation has densities of Casimirs as kernel. This is why Proposition \ref{prop:recursion relation for qHam} allows to renconstruct the entire integrable hierarchy.
﻿
\section{Quantum hierarchy for elliptic curves}

\subsection{Gromov-Witten classes}
Let $E$ be a non-singular complex elliptic curve. 
Its cohomology $H^*(E)$ is a $\mathbb{Z}_2$-graded vector space spanned by $\{e_1, e_2, e_3, e_4\}$, where
\begin{itemize}
    \item $e_1 \in H^0(E)$ is the unit.
    \item $e_2, e_3 \in H^1(E)$ are odd classes such that $\int_E e_2 \cup e_3 = 1$.
    \item $e_4 \in H^2(E)$ is the Poincaré dual of a point, $\int_E e_4 = 1$.
\end{itemize}
The moduli space of stable maps to $E$ of degree $d$ from the genus $g$ curves with $n$ marked points is denoted by $\overline{M}_{g,n}(E,d)$ and we assume that  $2g-2+n>0$. There are natural maps
\begin{itemize}
    \item $\mathsf{src}_d\colon  \overline{M}_{g,n}(E,d) \rightarrow \overline{M}_{g,n}$ forgets the stable map and retains the stabilization of the source curve;
    \item $\mathsf{ev}_i\colon  \overline{M}_{g,n}(E,d) \rightarrow E$ is the evaluation map at the $i$-th marked point, $i=1,\dots,n$.
\end{itemize}
The cohomological field theory $\{c_{g,n}\}$ associated with the Gromov-Witten theory of $E$ is defined over the ring of formal power series $\mathbb{C}[\![ q ]\!]$ by
\begin{equation*}
    c_{g,n}(e_{\alpha_1} \otimes \ldots \otimes e_{\alpha_n}) \coloneqq \sum_{d\geq 0} (\mathsf{src}_d)_{*}\left([\overline{M}_{g,n}(E,d)]^{\mathrm{vir}}\textstyle\prod_{i=1}^n \mathsf{ev}_i^*(e_{\alpha_i})\right) q^d \in H^*(\overline{M}_{g,n})\otimes \mathbb{C}[\![ q ]\!].
\end{equation*}

\subsection{Quasimodular forms}
For even \(k \geq 2\), the \(k\)-weighted Eisenstein series is given by
\begin{equation*}
    \mathsf{G}_k(q) = -\frac{B_k}{2k} + \sum_{n \geq 1} \sigma_{k-1}(n) q^n = -\frac{B_k}{2k} + 
    \sum_{d\geq 1} d^{k-1} \frac{q^d}{1-q^d} ,
\end{equation*}
where \(B_k\) are the Bernoulli numbers and \(\sigma_k(n) = \sum_{d \mid n} d^k\) is the divisor function. 
The algebra of quasimodular forms if defined as
\begin{equation*}
	\mathrm{QMod} = \mathbb{C}[\mathsf{G}_2, \mathsf{G}_4, \mathsf{G}_6].
\end{equation*}
It is graded by non-negative even integers, $\mathrm{QMod} = \bigoplus_{\ell=1}^\infty \mathrm{QMod}_{2\ell}$, where the grading (also called weight) is defined on the generators as $\deg_{\mathrm{QMod}} (\mathsf{G}_k) \coloneq k$, $k=2,4,6$.  Note that for all even \(k \geq 2\), $\mathrm{QMod}_k \ni \mathsf{G}_k$. 
﻿
There are two differential operators that are often used in computations with quasimodular forms:
\begin{align*}
    \frac{d}{d\mathsf{G}_2} & \colon \mathrm{QMod}_k \rightarrow \mathrm{QMod}_{k-2}; \\
    D_q \coloneqq q \frac{d}{dq} &\colon \mathrm{QMod}_k \rightarrow \mathrm{QMod}_{k+2}.
\end{align*}
Their commutator $\left[ \frac{d}{d\mathsf{G}_2}, D_q \right]$ acts on $\mathrm{QMod}_k$ as the operator of multiplication by $-2k$. 
﻿

﻿
\subsection{DR hierarchy potential}
Proposition~\ref{prop:recursion relation for qHam} and equation \eqref{eq: dilaton} imply that it is enough to know the DR hierarchy potential $\overline{G}$ to reconstruct the full system of quantum Hamiltonians from the recursion relations.\\

From equation \eqref{eq:DR potential}, to write down $\overline{G}$ we need to compute the following integrals:
\begin{align}\label{eq:integrals}
\int_{\overline{M}_{g,n}} \mathrm{DR}_{g}(A)\, \Lambda( \tfrac{-\varepsilon^2}{\cxi\hbar}) \,  c_{g,n}(e_{\alpha_{1}} \otimes \cdots \otimes e_{\alpha_{n}})
\end{align}
﻿
Let us collect some remarks that will restrict the number of non-trivial integrals we have to handle:
\begin{itemize}
	\item First of all, since the double ramification cycle $\mathrm{DR}_{g}\big( -\textstyle\sum_{i=1}^n a_{i}, A\big)$ is symmetric in $a_1,\dots,a_n$, while $c_{g,n}$ is graded-$S_n$-equivariant we have non-trivial contributions to~\eqref{eq:integrals} only when $c_{g,n}\big(\bigotimes_{i=1}^n e_{\alpha_i} \big)$ contains at most one $\alpha_i=2$ and at most one $\alpha_i=3$.
	\item Second, by~\cite{janda-targetcurves}, the number of $\alpha_i$'s equal to $2$ should be the same as the number of $\alpha_i$'s equal to $3$, otherwise $c_{g,n}\big(\bigotimes_{i=1}^n e_{\alpha_i}\big)$ vanishes. 
	\item Third, by~\cite{pixton_oberdieck_23}, the class $\lambda_g \, c_{g,n}\big(\bigotimes_{i=1}^n e_{\alpha_i}\big)$ vanishes for $g\geq 1$.
	\item Finally, the total cohomological degree of the class in \eqref{eq:integrals} must be equal to $2(3g-3+n)$, and the cohomological degrees of the involved classes are given by 
	\[
	\dim \mathrm{DR}_{g}(A) = 2g, \quad \dim \lambda_j = 2j, \quad \text{and} \quad  \dim c_{g,n}(e_{\alpha_{1}} \otimes \cdots \otimes e_{\alpha_{n}}) = 2g-2+\textstyle\sum_{i=1}^n \dim e_{\alpha_i}.
	\]
\end{itemize} 
Thus the potentially nontrivial integrals~\eqref{eq:integrals} are reduced to:
\begin{align} \label{eq:integrals g11 genus 0}
		& \int_{\overline{M}_{0,n}} c_{0,n}(e_4\otimes e_1^{\otimes (n-1)}), \qquad 
		\int_{\overline{M}_{0,n}} c_{0,n}(e_2\otimes e_3\otimes  e_1^{\otimes (n-2)}),  \\
		\label{eq:integrals g11 lambda g-1} 
	& \int_{\overline{M}_{g,n}} \mathrm{DR}_g(A)\, \lambda_{g-1}\, c_{g,n}(e_1\otimes e_4^{\otimes (n-1)}), \qquad  
		\int_{\overline{M}_{g,n}} \mathrm{DR}_g(A)\, \lambda_{g-1}\, c_{g,n}(e_2\otimes e_3\otimes e_4^{\otimes (n-2)}), \\
	\label{eq: integral G11 lambda g-2}
	& \int_{\overline{M}_{g,n}} \mathrm{DR}_g(A)\, \lambda_{g-2}\, c_{g,n}(e_4^{\otimes n}). 
\end{align}
By axiom (v) of cohomological field theories, the genus $0$ integrals~\eqref{eq:integrals g11 genus 0} vanish unless $n=3$, and for $n=3$ both integrals are equal to $1$, see e.g.~\cite{buryak_22}. The higher genera integrals with $\lambda_{g-1}~\eqref{eq:integrals g11 lambda g-1}$ are computed in~\cite{pixton_oberdieck_23}:
\begin{prop}[{\cite[Theorem 6.10]{pixton_oberdieck_23}}]
	\label{prop: omega-class}
	For $g\geq 1, n\geq 1$ and $A=(a_1,\ldots,a_n)\in \mathbb{Z}^n$, we have 
\begin{align*}
			\int _{\overline{M}_{g,n}} \mathrm{DR}_{g}( A)\, \lambda_{g-1}\, c_{g,n}\left( e_1\otimes e_{4}^{n-1}\right) & = \frac {a_{1}^{2}}{2^{2g-2}}\sum _{\sum b_{i} =g-1}\prod _{i=1}^{n}\frac{a_{i}^{2b_i}}{( 2b_{i} +1) !} \ D_{q}^{n-2} \mathsf{G}_{2g}
			\\
				\int _{\overline{M}_{g,n}} \mathrm{DR}_{g}( A) \, \lambda_{g-1}\, c_{g,n}\left( e_{2} \otimes e_3 \otimes e_{4}^{n-2}\right) & =\frac {-a_{1}a_{2}}{2^{2g-2}}\sum _{\sum b_{i} =g-1}\prod _{i=1}^{n}\frac{a_{i}^{2b_i}}{( 2b_{i} +1) !} \ D_{q}^{n-2} \mathsf{G}_{2g}
\end{align*}
\end{prop}

Finally, the integral~\eqref{eq: integral G11 lambda g-2} vanishes by the following Proposition. 

\begin{prop}\label{prop: vanishing}
Let $(a_1,\dots,a_n)\in\mathbb{Z}^n$ with $\sum_i a_i=0$ and $g\geq 2$.
Then
\[ 
\int_{\overline{M}_{g,n}} \mathrm{DR}_g(a_1,\ldots,a_n) \; \lambda_{g-2}\;c_{g,n}\!\big(e_4^{\otimes n}\big) =0.
\]
\end{prop}

\begin{proof} If $a_1=\cdots=a_n=0$, then  $\mathrm{DR}_g(a_1,\ldots,a_n) = (-1)^g\lambda_g$ and the desired vanishing follows from the vanishing of $\lambda_g \, c_{g,n}\big(\bigotimes_{i=1}^n e_{\alpha_i}\big)$ for $g\geq 1$ that we recall above.
	
Assume that not all $a_i$ vanish. Then we prove a more refined statement, using the following interpretation of the product $\mathrm{DR}_g(a_1,\ldots,a_n) \;c_{g,n}\!\big(e_4^{\otimes n}\big)$. As before, let $A_+$ denote the subtuple of $A$ consisting of all the positive integers in $A$, and let $A_-$ denote the subtuple of $A$ consisting of all the negative integers in $A$, and let $n_0$ be the number of $a_i$'s that are equal to $0$. Namely, let $\overline{M}^\sim_{g,n_0}(E\times \mathbb{P}^1, d, A_-, A_+)$ be the moduli space of stable relative maps of connected genus $g$ curves to the rubber trivial bundle over $E$, whose projection to $E$ has degree $d$. By relative we mean relative to the zero and infinity sections, with ramification profiles given by $A_-$ and $A_+$, respectively. Let
\begin{equation*}
	\mathsf{src} \colon \overline{M}^\sim_{g,n_0}(E
	\times \mathbb{P}^1, d, A_-, A_+) \rightarrow \overline{M}_{g,n}
\end{equation*}
be the source map that forgets the stable relative map and retains the stabilization of the source curve. Then we have
\begin{align*}
& \int_{\overline{M}_{g,n}} \mathrm{DR}_g(a_1,\ldots,a_n) \; \lambda_{g-2}\;c_{g,n}\!\big(e_4^{\otimes n}\big) 
\\ & =
\sum_{d=0}^\infty q^d \int_{\overline{M}_{g,n}} \mathsf{src}_* \left(
\left[\overline{M}^\sim_{g,n_0}(E
\times \mathbb{P}^1, d, A_-, A_+)\right]^{\mathrm{vir}} \prod_{i=1}^n \mathsf{ev}_i^*(e_4)\right) \lambda_{g-2}
,
\end{align*}
where $\mathsf{ev}_i\colon \overline{M}^\sim_{g,n_0}(E
\times \mathbb{P}^1, d, A_-, A_+) \to E$ is the evaluation map at the $i$-th marked point for the relative stable map combined with the projection to $E$. 
	
Now the statement of the proposition follows from the following lemma:
\begin{lemma} Assume that not all $a_i$ vanish. Then for any $g,d\geq 0$ we have 
\[
\left[\overline{M}^\sim_{g,n_0}(E
\times \mathbb{P}^1, d, A_-, A_+)\right]^{\mathrm{vir}} \prod_{i=1}^n \mathsf{ev}_i^*(e_4) = 0.
\]
\end{lemma}

The proof of the lemma is based on the following general observation. Let $F\colon C_1\to C_2$ be a holomorphic map of smooth compact curves and $\sum_{i=1}^n b_i p_i$ be a principal divisor on $C_1$. Let us prove that then $\sum_{i=1}^n b_i F(p_i)$ is a principal divisor on $C_2$. If $F$ is constant map, then the statment is obvious. Assume $F$ is non-constant and let $\sum_{i=1}^k b_i p_i = (f)$, where $f\colon C_1\to \mathbb{P}^1$ is a meromorphic function. Then $\sum_{i=1}^k b_i F(p_i) = (g)$, with $g\colon C_2\to \mathbb{P}^1$ defined as $g(q) \coloneqq \prod_{p\in F^{-1}(q)} f(p)^{\mathrm{ord}_p F}$, where $\mathrm{ord}_p F$ denotes the local ramification degree of $F$ at $p\in C_1$. 
	
This observation implies that for a choice of points $p_1,\dots,p_n\in E$ such that $\sum_{i=1}^n a_i p_i$ is not a principal divisor on $E$, the geometric intersection $\cap_{i=1}^n \mathsf{ev}_i^{-1}(p_i)$ on $\overline{M}^\sim_{g,n_0}(E
\times \mathbb{P}^1, d, A_-, A_+)$ is empty. Indeed, let $(C,x_,\dots,x_n,F,f)\in \overline{M}^\sim_{g,n_0}(E
\times \mathbb{P}^1, d, A_-, A_+)$ be a possibly nodal curve $C$ with marked points $x_1,\dots,x_n$ whose normalization has $l$ components $C_1,\dots,C_l$, equipped with a map $F$ to $E$ whose restriction to $C_j$ is denoted by $F_j$, $j=1,\dots,l$, and a relative stable map $f$ to the rubber $\mathbb{P}^1$, whose restriction to $C_j$ is denoted by $f_j$, $j=1,\dots,l$. Note that
\begin{itemize}
	\item The divisor of $f_j$ on $C_j$ is principal and supported on nodes and marked points. Applying the observation above to the map $F_j\colon C_j\to E$, we conclude that the image of the divisor of $f_j$ under the map $F_j$ is a principal divisor on $E$, $j=1,\dots,l$. 
	\item Each marked point $x_i$ enters the divisor of exactly one $f_j$ with coefficient $a_i$. Each node enters the divisors of exactly two $f_j$'s, with the opposite coefficients. Hence, the sum of these principal divisors is equal to $\sum_{i=1}^n a_i F(x_i) = \sum_{i=1}^n a_i \mathsf{ev}_i(C,x_,\dots,x_n,F,f)$.
\end{itemize}
Thus,  $\cap_{i=1}^n \mathsf{ev}_i^{-1}(p_i)$ is empty if  $\sum_{i=1}^n a_i p_i$ is not a principal divisor on $E$, which implies the statement of the lemma and, as a corollary, the statement of the proposition.
\end{proof}

Given two power series $f(\varepsilon)= \sum_{i\geq 0} a_i \varepsilon^i\ \in R[[\varepsilon]]$ and $g(\varepsilon)= \sum_{i\geq 0} b_i \varepsilon^i \in M[[\varepsilon]]$ with $R$ a ring and $M$ a left module over $R$, consider their Hadamard product $(f \odot g)(\varepsilon) := \sum_{i \geq 0} a_i b_i \varepsilon^i \in M[[\varepsilon]]$. Moreover, for any power series $h(\varepsilon_1,\varepsilon_2) = \sum_{i,j \geq 0} c_{i,j} \varepsilon_1^i \varepsilon_2^j \in M[[\varepsilon_1,\varepsilon_2]]$, we define its diagonal as the power series $\operatorname{diag}_{\varepsilon_1,\varepsilon_2}\left\{ h(\varepsilon_1,\varepsilon_2)\right\} := \sum_{i\geq 0} c_{i,i} \varepsilon_1^i \in M[[\varepsilon_1]]$.\\

\begin{prop}\label{prop:main}
The quantum double ramification hierarchy for the Gromov–Witten theory of the elliptic curve is uniquely determined by the DR potential
\begin{align*}
\overline{G} =\int \Big[ \frac{(u^1)^2 u^4}{2} + u^1 u^2 u^3 + \frac{\cxi \hbar}{\varepsilon^2} \Big(& \sum_{\substack{g \geq 1 \\ n\geq 2}} \frac{\epsilon^{2g}}{2^{2g-2}(n-2)!} \sum_{\sum_{i=1}^n b_i = g-1} \Big(\frac{u^1_{2b_1+1}}{(2b_1+1)!}\frac{u^4_{2b_2+1}}{(2b_2+1)!}\\
&+ \frac{u^2_{2b_1+1}}{(2b_1+1)!}\frac{u^3_{2b_2+1}}{(2b_2+1)!}\Big)\prod_{i=3}^n \frac{u^4_{2b_i}}{(2b_i+1)!} D_q^{n-2} \mathsf{G}_{2g}(q)\Big)\Big]dx.
\end{align*}
or, equivalently,
\begin{align*}
\overline{G} = &\int \left[ \frac{(u^1)^2 u^4}{2} + u^1 u^2 u^3 + \cxi \hbar \left(\left(\mathcal{S}_\varepsilon (u^1_x)\,\mathcal{S}_\varepsilon(u^4_x) + \mathcal{S}_\varepsilon(u^2_x)\,\mathcal{S}_\varepsilon(u^3_x)\right)\,\exp\left(\mathcal{S}_\varepsilon (u^4)D_q\right)\right) \odot \mathcal{G}(\varepsilon,q)\right]dx \\
= & \int \left[ \frac{(u^1)^2 u^4}{2} + u^1 u^2 u^3 + \cxi \hbar \operatorname{diag}_{\varepsilon,\varepsilon'}\left\{ \left(\mathcal{S}_\varepsilon ( u^1_x)\,\mathcal{S}_\varepsilon( u^4_x) + \mathcal{S}_\varepsilon( u^2_x)\,\mathcal{S}_\varepsilon( u^3_x)\right)\mathcal{G}\left(\varepsilon',q\exp\left( \mathcal{S}_\varepsilon (u^4)\right)\right)\right\} \right]dx,\\
\end{align*}
where
$$\mathcal{S_\varepsilon} := \frac{\sinh(\frac{1}{2}\varepsilon \partial_x)}{\frac{1}{2}\varepsilon \partial_x} ,\qquad \mathcal{G}(\varepsilon,q) := \sum_{g\geq 0} \varepsilon^{2g} \mathsf{G}_{2g+2}(q). $$
\end{prop}
\begin{proof}
The genus $0$ integrals~\eqref{eq:integrals g11 genus 0} vanish unless $n=3$, and for $n=3$ both integrals are equal to $1$ and by Proposition \ref{prop: omega-class} (exploiting $a_1 = -\sum_{i=2}^n a_i$ in the first type of intersection numbers) and Proposition \ref{prop: vanishing}, we compute formula \eqref{eq:DR potential} as
\begin{align*}
\overline{G}=\int \Big[ \frac{(u^1)^2 u^4}{2} + u^1 u^2 u^3 + \frac{\cxi \hbar}{\varepsilon^2} \Big(& \sum_{\substack{g \geq 1 \\ n\geq 2}} \frac{\epsilon^{2g}}{2^{2g-2}(n-2)!} \sum_{\sum_{i=1}^n b_i = g-1} \Big(\frac{u^1_{2b_1+1}}{(2b_1+1)!}\frac{u^4_{2b_2+1}}{(2b_2+1)!}\\
&+ \frac{u^2_{2b_1+1}}{(2b_1+1)!}\frac{u^3_{2b_2+1}}{(2b_2+1)!}\Big)\prod_{i=3}^n \frac{u^4_{2b_i}}{(2b_i+1)!} D_q^{n-2} \mathsf{G}_{2g}(q)\Big)\Big]dx.
\end{align*}

Next we remark that, for a formal variable $\delta$ and for $1\leq \alpha\leq 4$,
\begin{align*}
\sum_{b\geq 0}\frac{\delta^{2b}u^\alpha_{2b}}{2^{2b}(2b+1)!}  = \mathcal{S}_\delta(u^\alpha),
\end{align*}
so that
\begin{align*}
\sum_{\substack{g \geq 1 \\ n\geq 2}} \frac{\epsilon^{2g-2}}{2^{2g-2}(n-2)!}& \sum_{\sum_{i=1}^n b_i = g-1} \Big(\frac{u^1_{2b_1+1}}{(2b_1+1)!}\frac{u^4_{2b_2+1}}{(2b_2+1)!}
+ \frac{u^2_{2b_1+1}}{(2b_1+1)!}\frac{u^3_{2b_2+1}}{(2b_2+1)!}\Big)\prod_{i=3}^n \frac{u^4_{2b_i}}{(2b_i+1)!} D_q^{n-2} \mathsf{G}_{2g}(q)\\
& = \sum_{\substack{g \geq 1 \\ n \geq 2}}\frac{\varepsilon^{2g-2}}{(n-2)!}\operatorname{Coeff}_{\delta^{2g-2}}\left[\left(\mathcal{S}_\delta ( u^1_x)\,\mathcal{S}_\delta( u^4_x) + \mathcal{S}_\delta( u^2_x)\,\mathcal{S}_\delta( u^3_x)\right)\,\left(\mathcal{S}_\delta (u^4)\, D_q\right)^{n-2}\right]\mathsf{G}_{2g}(q)\\
& =\sum_{g \geq 1}\varepsilon^{2g-2} \operatorname{Coeff}_{\delta^{2g-2}}\left[\left(\mathcal{S}_\delta ( u^1_x)\,\mathcal{S}_\delta( u^4_x) + \mathcal{S}_\delta( u^2_x)\,\mathcal{S}_\delta( u^3_x)\right)\,\exp\left(\mathcal{S}_\delta (u^4)D_q\right)\right]\mathsf{G}_{2g}(q).
\end{align*}
The above power series is the Hadamard product
\begin{align*}
\left[\left(\mathcal{S}_\varepsilon (u^1_x)\,\mathcal{S}_\varepsilon( u^4_x) + \mathcal{S}_\varepsilon(u^2_x)\,\mathcal{S}_\varepsilon( u^3_x)\right)\,\exp\left(\mathcal{S}_\varepsilon\left(u^4\right)D_q\right)\right] \odot \mathcal{G}(\varepsilon,q).
\end{align*}
In the second line of the statement, we have made use of the following identity, valid for any differential polynomial $f$:
\begin{align*}
e^{f D_q} \mathcal{G}(\varepsilon,q) = \mathcal{G}(\varepsilon,e^f q).
\end{align*}
\end{proof}

\begin{remark} Notice that, since the power series whose coefficients are the Bernoulli numbers is divergent, so is the series $\mathcal{G}(\varepsilon,q)$ as a power series in $\varepsilon$. This means that, in the above formulae, $\mathcal{G}(\varepsilon,q)$ must be treated as a formal power series only. With the idea of finding an expression for $\overline{G}$ which avoids the use of divergent power series, let us recall some facts about the Weierstra\ss{} elliptic function and its relation with the Eisenstein series $G_k$, see for instance \cite{pixton_oberdieck_18}. In the ring $0<|q|<|p|<1$, define the shifted Weierstra\ss{} elliptic function as
$$\overline{\wp}_\tau(z):=\sum_{a\in \mathbb{Z}_{\not= 0}} \frac{a}{1-q^a} p^a ,\qquad \text{where}\ q=\exp(2\pi\cxi \tau),\quad p=\exp(2\pi\cxi z).$$
The shifted Weierstra\ss{} function $\overline{\wp}_\tau(z)$ expands at $z=0$ as 
	\begin{align*} \label{eq:wp-expansion}
		\frac 1{(2\pi\cxi z)^2} + 2 \sum_{\ell=0}^\infty G_{2\ell+2}(q) \frac{(2\pi\cxi z)^{2\ell}}{(2\ell)!}.		
	\end{align*}
The Weierstra\ss{} elliptic function $\wp_\tau(z)$ is then defefined by removing the constant term from $\overline{\wp}_{\tau}(z)$, i.e. $\overline{\wp}_\tau(z)\coloneqq {\wp}_\tau(z) + 2G_2(q)$. 

Next, consider the following elementary fact about residues and the Laplace transform:
\begin{align*}
\frac{1}{2\pi\cxi}  \oint_{|\zeta|=1} \int_{\delta\in\mathbb{R}_{\geq 0}} \delta^k\zeta^\ell e^{-\delta\zeta} d\delta\, d\zeta \coloneqq \delta_{k\ell} \ell! \, , \qquad k\in \mathbb{Z}_{\geq 0}, \ell \in \mathbb{Z}.
\end{align*}

We combine these two observations to rewrite the formula for $\overline{G}$ in the following equivalent way: 
\begin{align*}
\overline{G}& =  \int \left[ \frac{(u^1)^2 u^4}{2} + u^1 u^2 u^3 + \frac{\hbar}{4\pi} \oint\!\!\int
\left(\mathcal{S}_{\delta} ( u^1_x) \,\mathcal{S}_{\delta}( u^4_x) +\mathcal{S}_{\delta}( u^2_x)\,\mathcal{S}_{\delta}( u^3_x)\right) \overline{\wp}_{\tau+\frac{\mathcal{S}_{\delta} (u^4)}{2\pi\cxi}}\Big(\frac{\zeta}{2\pi\cxi}\Big) \frac{e^{-\delta\zeta /\varepsilon}}{\varepsilon}	d\delta d\zeta \right]dx,
\end{align*}
where in the choise of the countour one has to assume that either $|\zeta/\varepsilon|=1$ and $\delta \in \mathbb{R}_{\geq 0}$, or, alternatively, $|\zeta|=1$ and $\delta/\varepsilon \in \mathbb{R}_{\geq 0}$.

It is also interesting to compute the primary Hamiltonians of the DR hierarchy, $\overline{G}_{\alpha,0} = \frac{\partial \overline{G}}{\partial u^\alpha}$, $1 \leq \alpha\leq 4$, as
\begin{align*}
\overline{G}_{1,0} & = \int \left(u^1 u^4 + u^2 u^3\right) dx, \\
\overline{G}_{2,0} & = \int \left( u^1 u^3\right) dx,\\
\overline{G}_{3,0} & = \int \left(-u^1 u^2 \right) dx,\\
\overline{G}_{4,0} & = \int \left[ \frac{(u^1)^2}{2} - \frac{\cxi \hbar}{8\pi^2}
\oint\!\!\int \left(\mathcal{S}_{\delta} ( u^1_x) \,\mathcal{S}_{\delta}( u^4_x) + \mathcal{S}_{\delta}( u^2_x)\,\mathcal{S}_{\delta}( u^3_x)\right)
	\partial_\tau \overline{\wp}_{\tau+\frac{\mathcal{S}_{\delta} (u^4)}{2\pi\cxi}}\Big(\frac{\zeta}{2\pi\cxi}\Big) \frac{e^{-\delta\zeta / \varepsilon}}{\varepsilon}
	d\delta\, d\zeta \right] dx.
\end{align*}
The last Hamiltonian can be written as well as
\begin{align*}
\overline{G}_{4,0} &= \int \left[ \frac{(u^1)^2}{2} - \frac{\cxi \hbar}{8\pi^2} 
\oint\!\!\int
\left(\Delta_{\delta} ( u^1) \,\Delta_{\delta}( u^4) + \Delta_{\delta}( u^2)\,\Delta_{\delta}( u^3)\right)
	\partial_\tau \overline{\wp}_{\tau+\frac{\Delta_{\delta} (\partial_x^{-1}u^4)}{2\pi\cxi}}\Big(\frac{\zeta}{2\pi\cxi}\Big) \frac{e^{-\delta\zeta / \varepsilon}}{\varepsilon}
	d\delta\, d\zeta \right] dx,
\end{align*}
where $\Delta_y (u^\alpha)(x)= \frac{2}{y}\left(u^\alpha(x+y/2) - u^\alpha(x-y/2)\right)$, $\tau=\frac{1}{2\pi \cxi} \log q$.\\
\end{remark}

\begin{remark}
There are several interesting limits for the potential computed in Proposition \ref{prop:main}. First, the \emph{dispersionless limit} $\varepsilon \to 0$,
\begin{align*}
\left.\overline{G}\right|_{\varepsilon=0} = \int \left[ \frac{(u^1)^2 u^4}{2} + u^1 u^2 u^3 + \cxi \hbar\,  (u^1_x u^4_x +  u^2_x u^3_x)\, \mathsf{G}_{2} \left( q e^{u_4}\right)\right]dx.\\
\end{align*}
Second, the \emph{trigonometric} $q \to 0$ limit, or $\tau \to +\cxi \infty$. Since $\mathsf{G}_{2g}(q) \to -\frac{B_{2g}}{4g}$ as $q\to 0$ we have
$$\lim_{\tau \to +\cxi \infty} \overline{\wp}_{\tau}\left(\frac{\varepsilon \zeta}{2 \pi \cxi}\right) =-\frac{1} {4 \sin^2\left(\frac{\varepsilon \zeta}{2\cxi}\right)}$$
and we obtain
\begin{align*}
\left.\overline{G}\right|_{q=0} = \int \left[ \frac{(u^1)^2 u^4}{2} + u^1 u^2 u^3 + \frac{\hbar}{4\pi} 
\oint\!\!\int
\frac{\mathcal{S}_{\delta} ( u^1_x) \,\mathcal{S}_{\delta}( u^4_x) +\mathcal{S}_{\delta}( u^2_x)\,\mathcal{S}_{\delta}( u^3_x)}{4 \sin^2\left(\frac{\varepsilon \zeta}{2\cxi}\right)} e^{-\delta\zeta}	d\delta d\zeta \right]dx.\\
\end{align*}
Third, the \emph{double scaling classical limit} obtined by replacing
\begin{equation}\label{eq:dsl}
\begin{array}{lll}
u^1_k \mapsto a^{-1} u^1_k, \qquad &  u^2_k \mapsto a^{-1} u^2_k, \qquad& k\geq 0,\\
u^3_k \mapsto a u^3_k, \qquad &  u^4_k \mapsto a u^4_k, \qquad &k\geq 0,\\
\tau \mapsto a \tau, \qquad & \varepsilon \mapsto a \varepsilon, \qquad & \hbar \mapsto a \mu^2
\end{array}
\end{equation}
in the rescaled potential $a \overline{G}$, with $\mu$ a new formal variable, and then taking the $a \to 0^+$ limit, along the real axis, with $\tau$ fixed in the upper half plane. Because $\mathsf{G}_{2g}(e^{2\pi \cxi a \tau}) \sim -\frac{B_{2g}}{4g}(a\tau)^{-2g}$ as $a\to 0^+$, we have
$$\lim_{a \to 0^+} a^2 \overline{\wp}_{a \tau}\left(\frac{a \varepsilon \zeta}{2 \pi \cxi}\right) =-\frac{1} {4 \tau^2 \sin^2\left(\frac{\varepsilon \zeta}{2\cxi \tau}\right)}$$
and we obtain a new \emph{classical} DR potential
\begin{align*}
\overline{h} =  \int \left[ \frac{(u^1)^2 u^4}{2} + u^1 u^2 u^3 + \frac{\mu^2}{4\pi} 
\oint\!\!\int
\frac{\mathcal{S}_{\delta} ( u^1_x) \,\mathcal{S}_{\delta}( u^4_x) +\mathcal{S}_{\delta}( u^2_x)\,\mathcal{S}_{\delta}( u^3_x)}{4\left(\tau+\frac{\mathcal{S}_{\delta} (u^4) }{2 \pi \cxi}\right)^2 \sin^2\left(\frac{\varepsilon \zeta}{2\cxi\left(\tau+\frac{\mathcal{S}_{\delta} (u^4) }{2 \pi \cxi}\right)}\right)} e^{-\delta\zeta}	d\delta d\zeta \right]dx.
\end{align*}
In particular, the Hamiltonian densities $h_{\alpha,d}$, $1\leq\alpha\leq 4$, $d\geq -1$, obtained, in the $a \to 0$ limit, by the replacements \eqref{eq:dsl} in the rescaled DR Hamiltonian densities $a^d G_{1,d}$, $a^d G_{2,d}$, $a^{d+2}G_{3,d}$, $a^{d+2}G_{4,d}$, respectively, satisfy the DR recursion
$$\partial_x (D-1) h_{\alpha,d+1} = \{h_{\alpha,d},h_{1,1}\}, \qquad 1\leq\alpha\leq 4, d\geq -1,$$
where $D=\mu\frac{\partial}{\partial \mu} + \varepsilon \frac{\partial}{\partial \varepsilon}  + \sum_{s=0}^\infty u^{\alpha}_s \frac{\partial}{\partial u^{\alpha}_s}$, $h_{1,-1}=u^4$, $h_{2,-1}=u^3$, $h_{3,-1}=-u^2$, $h_{4,-1}=u^1$, $\{\cdot,\cdot\} = \left.\left(\frac{1}{\hbar}[\cdot,\cdot]\right)\right|_{\hbar=0}$ and $\overline{h}_{1,1} = (D-2) \overline{h}$.\\
Further sending $\varepsilon \to 0$ (i.e. taking the double scaling limit of the quantum dispersionless limit above) we obtain
\begin{align*}
\left.\overline{h}\right|_{\varepsilon=0} = \int \left[ \frac{(u^1)^2 u^4}{2} + u^1 u^2 u^3  +\frac{\cxi \pi^2 \mu^2}{6}\,  \frac{(u^1_x u^4_x +  u^2_x u^3_x)}{\left(2 \pi \cxi\tau+u^4\right)^2}\, \right]dx.\\
\end{align*}
which produces, in particular, the Hamiltonians
$\overline{h}_{4,0} = \int\left[\frac{(u^1)^2}{2} -\frac{\cxi \pi^2 \mu^2}{3}\,  \frac{(u^1_x u^4_x +  u^2_x u^3_x)}{\left(2 \pi \cxi\tau+u^4\right)^3}\, \right]dx$ and $\overline{h}_{1,1}=\int \left[ \frac{(u^1)^2 u^4}{2} + u^1 u^2 u^3  -\frac{2 \pi^3 \mu^2 \tau}{3}\,  \frac{(u^1_x u^4_x +  u^2_x u^3_x)}{\left(2 \pi \cxi\tau+u^4\right)^3}\, \right]dx$.
\end{remark}

\printbibliography

\end{document}